\documentclass[12pt]{amsart}

\usepackage{verbatim, amssymb, enumitem, mathtools,mathrsfs}
\usepackage{color}

\usepackage[breaklinks=true,colorlinks=true,linkcolor=blue,citecolor=red,urlcolor=blue,psdextra,pdfencoding=auto]{hyperref}
\allowdisplaybreaks

\renewcommand \a{\alpha}
\renewcommand \b{\beta}

\newcommand \la{\lambda}
\newcommand \ve{\varepsilon}
\newcommand \id{\mathrm{id}}
\newcommand \br{\mathbb{R}}
\newcommand \bc{\mathbb{C}}
\newcommand \bh{\mathbb{H}}
\newcommand \bo{\mathbb{O}}

\newcommand \rk{\operatorname{rk}}
\newcommand \Ker{\operatorname{Ker}}

\newcommand \End{\operatorname{End}}

\renewcommand \Re{\operatorname{Re}}
\newcommand \Span{\operatorname{Span}}
\newcommand \Tr{\operatorname{Tr}}

\newcommand \SO{\mathrm{SO}}
\newcommand \SU{\mathrm{SU}}
\newcommand \SL{\mathrm{SL}}
\newcommand \Sp{\mathrm{Sp}}
\newcommand \Ff{\mathrm{F}_4}
\newcommand \Spin{\mathrm{Spin}}

\newcommand \cJ{\mathcal{J}}

\newcommand \cI{\mathcal{I}}

\newcommand \cV{\mathcal{V}}

\newcommand \cH{\mathcal{H}}
\newcommand \cL{\mathcal{L}}
\newcommand \cT{\mathcal{T}}
\newcommand \cK{\mathcal{K}}
\newcommand \cM{\mathcal{M}}
\newcommand \cS{\mathcal{S}}

\newcommand \sK{\mathsf{K}}
\newcommand \sS{\mathsf{S}}
\newcommand \sL{\mathsf{L}}
\newcommand \sT{\mathsf{T}}
\newcommand \sW{\mathsf{W}}

\newcommand \bg{\mathfrak{b}}

\newcommand\g{\mathfrak g}

\newcommand\h{\mathfrak h}

\newcommand\m{\mathfrak m}
\newcommand \so{\mathfrak{so}}
\newcommand \spg{\mathfrak{sp}}

\newcommand \ug{\mathfrak{u}}
\newcommand \su{\mathfrak{su}}

\newcommand \f{\mathfrak{f}}
\newcommand \p{\mathfrak{p}}
\renewcommand\t{\mathfrak t}

\newcommand \ir{\mathrm{i}}

\newcommand \Sym{\operatorname{Sym}}

\newcommand \ad{\operatorname{ad}}

\newcommand \diag{\operatorname{diag}}

\newcommand \pd{\partial}
\newcommand \rL{\mathrm{L}}
\newcommand \rR{\mathrm{R}}

\newcommand \<{\langle}
\renewcommand \>{\rangle}
\newcommand \ip{\<\cdot,\cdot\>}

\newtheorem{theorem}{Theorem}
\newtheorem*{theorem*}{Theorem}
\newtheorem{corollary}{Corollary}
\newtheorem*{corollary*}{Corollary}
\newtheorem*{conj*}{Conjecture}
\newtheorem{lemma}{Lemma}
\newtheorem{proposition}{Proposition}
\newtheorem*{prop*}{Proposition}

\theoremstyle{definition}

\newtheorem*{definition*}{Definition}

\theoremstyle{remark}
\newtheorem{remark}{Remark}

\newtheorem*{notation*}{Notation}
\newtheorem*{algorithm*}{Algorithm}
\newtheorem*{example*}{Example}

\makeatletter
\@namedef{subjclassname@2020}{%
  \textup{2020} Mathematics Subject Classification}
\makeatother

\begin{document}

\title{On Killing tensors on Riemannian symmetric spaces}

\author{Vladimir Matveev}
\address{Fakult\"{a}t f\"{u}r Mathematik und Informatik, Friedrich-Schiller-Universit\"{a}t, 07737 Jena, Germany}
\email{vladimir.matveev@uni-jena.de}
\thanks{The first named author is thankful to La Trobe University for hospitality. \\ \hspace*{9.8pt} The first named  author was partially supported by the DFG (projects 455806247 and 529233771).} 

\author{Yuri Nikolayevsky}
\address{Department of Mathematical and Physical Sciences, La Trobe University, VIC 3086, Australia} 
\email{Y.Nikolayevsky@latrobe.edu.au}
\thanks{The second named author is thankful to Friedrich-Schiller-Universit\"{a}t for hospitality. \\ \hspace*{9.8pt} The authors were partially supported by ARC Discovery Grant DP210100951.} %




\begin{abstract}
A Killing tensor field on a Riemannian space corresponds to an integral of the geodesic flow polynomial in momenta. A Killing tensor field is called \emph{decomposable} if it is a polynomial in Killing vector fields. In this paper, we first prove that the study of Killing tensor fields on symmetric spaces can be reduced to the case of compact irreducible ones. Then we introduce the class of \emph{top slot} Killing tensor fields.
We obtain an explicit and elegant description of such tensor fields and prove that the quadratic Killing tensor fields are spanned by the top-slot ones. We also show that quadratic Killing tensor fields on the quaternionic projective space and on the Cayley projective space are spanned by the indecomposable ones constructed in our earlier paper and the decomposable ones. This completes the classification of quadratic Killing tensor fields on Riemannian symmetric spaces of rank one.
\end{abstract} 

\maketitle

\section{Introduction}
\label{s:intro}

A \emph{covariant Killing tensor field} $L=L(x)_{i_1\dots i_d}$ of rank $d \ge 1$ on a Riemannian manifold $(M,ds^2=g_{ij}dx^i dx^j)$ is a covariant symmetric tensor field that satisfies the Killing equation
\begin{equation}\label{eq:defK}
  L_{(i_1\dots i_d,j)}=0,
\end{equation}
where the comma denotes the covariant derivative and the parentheses denote the symmetrisation by all indices. This definition is equivalent to the fact that the function $\xi\in T_xM \mapsto L(x)_{i_1\dots i_d} \xi^{i_1} \cdots \xi^{i_d}$ polynomial in the velocities is an integral of the geodesic flow of $(M,ds^2)$: for any naturally parameterised geodesic $s \mapsto \gamma(s)$ of $(M,ds^2)$, the function $s \mapsto L(\gamma(s))_{i_1\dots i_d} (\dot{\gamma}(s))^{i_1} \dots (\dot{\gamma}(s))^{i_d}$ is constant. Another equivalent definition is as follows. A contravariant symmetric tensor field $K$ of rank $d$ defines a function $\cK$ on the cotangent bundle $T^*M$ which is a homogeneous polynomial of degree $d$ in the momenta $p_i$. For the contravariant metric tensor $g^{ij}$, this construction gives twice the Hamiltonian $\cH= \frac12 \sum_{ij} g^{ij}p_ip_j$. A contravariant symmetric tensor field $K$ is Killing if and only if the corresponding function $\cK$ on $T^*M$ \emph{Poisson commute} with the Hamiltonian (see Subsection~\ref{ss:Kt}). It is well known that a contravariant tensor field $K$ is Killing if and only if the corresponding covariant tensor field $K^\flat$ is.

Contravariant Killing tensor fields of rank $d=1$ are called \emph{Killing vector fields}. It is well known that a vector field is Killing if and only if the $1$-parametric group of diffeomorphisms of $M$ which it generates consists of isometries. We know no such nice geometric description for Killing tensor fields of rank $d \ge 2$.

The space of all contravariant Killing tensor fields on $(M,ds^2)$ forms an associative, commutative, graded algebra $\sK(M)$ with respect to the symmetric tensor product. It contains a subalgebra $\sS(M)$ generated by Killing vector fields. The elements of $\sS(M)$ are called \emph{decomposable}.

A general Riemannian manifold may have indecomposable Killing tensor fields, even if one disregards polynomials in the metric tensor. A classical example is due to Darboux: the $2$-dimensional Liouville metric $ds^2=(\la(x)+\mu(y)) (dx^2 + dy^2)$ admits quadratic Killing tensor fields not proportional to the metric, but, in general, has trivial isometry group. On the other hand, any Killing tensor field on a space of constant curvature is decomposable~\cite{Tho,ST1,Tak}, and the structure of the algebra of Killing tensor fields is completely understood.

In this paper, we discuss the structure of the space of Killing tensor fields on a Riemannian symmetric space. This is the next natural step, as symmetric spaces are the next class of Riemannian spaces after spaces of constant curvature in terms of complexity. Symmetric spaces have a very large isometry group, and hence a very rich algebra of decomposable Killing tensor fields. Moreover, a symmetric space can be viewed as a test laboratory for a general theory of Killing tensors, viewed from the perspective of the geometric theory of PDEs. Indeed, the equation \eqref{eq:defK} is of finite type and closes after $d$ prolongations; see, e.g.,~\cite[Theorem~4.3]{Tho} and \cite{EL,MSS}. Therefore, a Killing tensor field of rank $d$ is completely determined by its $d$-th jet at a single point. Further prolongations give the integrability conditions, which form a system of algebraic relations on the $d$-jet whose coefficients depend on the curvature of the metric and its covariant derivatives. In the constant curvature case, these algebraic relations vanish identically \cite{Tho}, and for a general metric they are almost unmanageable for $d>2$; see, e.g.,~\cite[Appendix~A]{WT}. However, on locally symmetric spaces, the covariant derivatives of the curvature vanish, so the integrability conditions depend only on the curvature tensor and can be calculated explicitly; see, e.g., Theorem \ref{t:quadratic} and Proposition \ref{p:top}.

Our results are motivated by and give a partial answer to \cite[Question~3.9]{BMMT}: ``\emph{Is every Killing tensor field on a symmetric space decomposable?}''
At present, we know that the answer to this question is in the positive for the following symmetric spaces:
\begin{itemize}
  \item
  all Killing tensor fields on the spaces of constant curvature~\cite{Tho,ST1,Tak};

  \item
  all Killing tensor fields on the complex projective spaces~\cite[Corollary~5]{East}, \cite[Theorem~2.2]{ST2};

  \item
  quadratic Killing tensor fields on some classical symmetric spaces: the classical groups $\SO(n)$~\cite{MNN}, the real Grassmannians $\SO(p+q)/\SO(p) \times \SO(q)$, the spaces $\SU(n)/\SO(n)$~\cite{NN} and the space $\SU(6)/\Sp(3)$~\cite{EL}.

  \item
  quadratic Killing tensor fields on some exceptional symmetric spaces: the group $\mathrm{G}_2$ and the space $\mathrm{G}_2/\SO(4)$~\cite{NN}.
\end{itemize}

Perhaps somewhat unexpectedly, the answer is in the negative already for quadratic Killing tensor fields on
\begin{itemize}
  \item
  the Cayley projective plane $\bo P^2 = \Ff/\Spin(9)$ \cite{MN1}.

  \item
  the quaternionic projective spaces $\bh P^n=\Sp(n+1)/(\Sp(n)\Sp(1))$ with $n \ge 3$ \cite{MN1}.

  \item
  the space $\mathrm{E}_6/\mathrm{F}_4$~\cite{EL}. 
\end{itemize}

\smallskip
An additional motivation for studying the decomposability of Killing tensors on symmetric and homogeneous spaces comes from Mishchenko--Fomenko type problems. For a homogeneous metric, one seeks a complete Poisson-commutative subalgebra of integrals, thereby obtaining Liouville integrability; see, e.g., \cite{BJ,MF}, and \cite{Rybnikov} for quantum analogues. It is natural to seek such commuting integrals within a prescribed class, usually the class of integrals polynomial in the momenta, since the Mishchenko--Fomenko problem is trivial in the $C^\infty$ category \cite[\S 6]{BFIMZ}. From this perspective, indecomposable Killing tensors are especially valuable, as they may lead to Liouville-integrable systems with integrals of lower degree in the momenta than those produced by standard methods.

\smallskip

Our paper  contains the following results.  
First, we establishing several important structural properties of the algebra of Killing tensor fields (see Subsection~\ref{ss:Kt} for unexplained notation).

\begin{theorem} \label{t:prod}
  Let $(M,g)=(M_0,g_0) \times (M_1,g_1) \times \dots \times (M_m,g_m)$ be the de Rham decomposition of a connected, simply connected, globally symmetric space $(M,g)$, where $(M_0,g_0)$ is Euclidean and $(M_i,g_i)$ are irreducible for $i=1, \dots, m$. Then $\sK(M) = \sK(M_0) \otimes \sK(M_1) \otimes \dots \otimes \sK(M_m)$.
\end{theorem}

It is important to emphasize that for general Riemannian manifolds, even if one prohibits locally flat factors, the algebra of Killing tensor fields on the product may be strictly bigger than the tensor product of the algebras of Killing tensor fields on the factors. An explicit example is given in~\cite[Section~4]{MN2}, and the structure of the algebra of Killing tensor field on the Riemannian product is described in~\cite[Theorem~1.3]{MN2}.

Irreducible symmetric spaces come in pairs of the duals. We establish the following.

\begin{theorem} \label{t:dual}
  Let $(M,g)$ and $(M',g')$ be the dual connected, simply connected, globally symmetric spaces. Then the complexifications $\sK(M) \otimes \bc$ and $\sK(M') \otimes \bc$ are isomorphic complex, associative, commutative, graded algebras.
\end{theorem}

In view of Theorem~\ref{t:prod} and~\ref{t:dual} and the fact that for the flat spaces everything is known, the study of Killing tensor fields on symmetric spaces reduces to the study for compact irreducible ones.

\smallskip

Next we focus on \emph{quadratic} Killing tensor fields. Recall that a Killing tensor field of rank $d$ is uniquely determined by its $d$-th jet at a point~\cite[Theorem~4.3]{Tho}. We say that a Killing tensor field at a point $x \in M$ is \emph{top slot}, if its $(d-1)$-st jet at that point is zero.

The following theorem gives a comprehensive description of quadratic Killing tensor fields on a symmetric space.

\begin{theorem} \label{t:quadratic}
  Let $(M,g)$ be a compact irreducible symmetric space, and let $\sL^2_x$ be the space of the top slot quadratic contravariant Killing tensor fields at the point $x \in M$. The following holds.
  \begin{enumerate}[label=\emph{(\alph*)},ref=\alph*]

    \item \label{it:quadraticspan}
    The space $\sK^2(M)$ of quadratic contravariant Killing tensor fields is spanned by the union $\cup_{x \in M} \sL^2_x$.

    \item \label{it:quadratictop1} 
    Relative to geodesic normal coordinates $X=(x^i)$ at a point $o \in M$ and the corresponding momenta $P=(p_i)$, the function $\cK$ on $T^*M$ corresponding to a Killing tensor field $K \in \sL^2_o$ is given by $\cK(X,P) = K_2(X,X,P,P)$, where $K_2$ is a constant tensor of type $(0,4)$ on $T_oM$ which is symmetric by its first and by its second pairs of arguments.

    \item \label{it:quadraticeqs}
    Conversely, let $K$ be a quadratic contravariant tensor field on $M$ such that the Taylor expansion at $o \in M$ of the function $\cK$ relative to geodesic normal coordinates is given by $\cK(X,P)=K_2(X,X,P,P)$, where $K_2$ is as in~\eqref{it:quadratictop1}.

    Then $K$ is Killing if and only if it satisfies the following system of linear equations: for all $X, P \in T_oM$,
    \begin{gather}
        K_2(X,X,X,P)=0,  \label{eq:quadraticsigma}\\
        K_2(X,X,P,R(X,P)P)=K_2(P,P,X,R(P,X)X), \label{eq:quadratic21}\\
        K_2(X,X,R(X,P)P,R(X,P)P)=K_2(P,P,R(P,X)X,R(P,X)X), \label{eq:quadratic22},
    \end{gather}
    where $R$ is the curvature tensor of $(M,g)$ at the point $o$.
  \end{enumerate}
\end{theorem}

\begin{remark} \label{rem:topslothigh}
  The picture is even more intriguing: it turns out that for the top slot Killing tensor fields of \emph{arbitrary} rank, in their Taylor expansion, not only the terms of degree fewer than $d$ are absent (which is the definition), but also the terms of degree greater than $d$. Hence top slot Killing tensor fields are defined by a single constant tensor of type $(d,d)$ on $\m$, similar to Theorem~\ref{t:quadratic}\eqref{it:quadratictop1} for the quadratic ones. What is more, necessary and sufficient conditions for such a constant tensor to define a top slot Killing field given in assertion~Theorem~\ref{t:quadratic}\eqref{it:quadraticeqs} remain valid for an arbitrary rank as well, with obvious modifications. We prove these facts in Proposition~\ref{p:top}.
  However, we do not know if the claim of~Theorem~\ref{t:quadratic}\eqref{it:quadraticspan} remains true for higher rank, that is, whether the space of Killing tensor fields of rank $d > 2$ is spanned by the top slot ones.
    
  In the language of PDE, this means that there are no further differential or algebraic obstructions for the local existence of a Killing tensor field (of rank $d \le 2$) and of a top slot Killing tensor field (of an arbitrary rank) then those given in Theorem~\ref{t:quadratic}\eqref{it:quadraticeqs} and in Proposition~\ref{p:top}\eqref{it:topint}, respectively.
\end{remark}


We apply Theorem~\ref{t:quadratic} to complete the description of quadratic Killing tensor fields on rank one symmetric spaces. By Theorem~\ref{t:dual}, we can restrict ourselves to the compact ones, and moreover, all the Killing tensor fields on the sphere and on $\bc P^n$ are known -- see above. On the other hand, there is a family of indecomposable quadratic tensor fields on the Cayley projective plane $\bo P^2$, and for every $n \ge 3$, a family of indecomposable quadratic tensor fields on the quaternionic projective space $\bh P^n$ constructed in~\cite{MN1}; we give their description in Section~\ref{s:hpnop2}. We prove that there no others:

\begin{theorem} \label{t:hpnop2}
{\ }
  \begin{enumerate}[label=\emph{(\alph*)},ref=\alph*]

    \item \label{it:hpnop2hp2}
    Any quadratic Killing tensor field on $\bh P^2$ is decomposable.

    \item \label{it:hpnop2hpn}
    The space of quadratic Killing tensor fields on $\bh P^m, \, m \ge 3$, is spanned by the decomposable ones and by those constructed in~\cite[Section~2]{MN1}.

    \item \label{it:hpnop2op2}
    The space of quadratic Killing tensor fields on $\bo P^2$ is spanned by the decomposable ones and those constructed in~\cite[Section~3]{MN1}.
  \end{enumerate}
\end{theorem}

We note that Theorem~\ref{t:hpnop2}(\ref{it:hpnop2op2}) can also be derived indirectly from \cite[\S~4.7]{EL}.

\section{Preliminaries}
\label{s:prel}

\subsection{Killing tensor fields}
\label{ss:Kt}

Let $(M,g)$ be a Riemannian manifold of dimension $n$. Let $(x^1, \dots, x^n)$ be local coordinates defined on a connected neighbourhood $U \subset M$, and let $(p_1, \dots, p_n)$ be the corresponding momenta, the functions on $T^*M$ defined by $p_i(x,\omega) = \omega(\partial/\partial x_i)$ for $x \in M$ and $\omega \in T^*_xM$.

The \emph{Poisson bracket} of differentiable functions $f$ and $h$ on $T^*M$, is defined by
\begin{equation}\label{eq:Poissonbr}
  \{f, h\} = \sum_{j=1}^{n} \Big(\frac{\pd f}{\pd x^i} \frac{\pd h}{\pd p_i} - \frac{\pd h}{\pd x^i} \frac{\pd f}{\pd p_i}\Big).
\end{equation}
The \emph{Hamiltonian} is the function on $T^*M$ defined by $\cH = \frac12 g^{ij} p_i p_j$. 

The curvature tensor is defined by $R(X,Y,Z,V)=\<R(X,Y)Z,V\>=\<(\nabla_X \nabla_Y - \nabla_Y \nabla_X - \nabla_{\nabla_XY - \nabla_YX})Z, V\>$, where $\nabla$ is the Levi-Civita connection and $\ip$ is the inner product defined by $g$. For $x \in M$ and $X \in T_xM$, the \emph{Jacobi operator} $R_X$ is a symmetric operator on $T_xM$ is defined by $R_X Y =R(Y,X)X$.

For $d \ge 1$, let $\sK^d(M)$ be the space of contravariant Killing tensor fields on $(M,g)$ of rank $d$; we formally define $\sK^0(M) = \br$. The following facts are well known (most of them can be found e.g. in~\cite{Tho}).
\begin{enumerate}[label={Fact \arabic*.},ref=\arabic*]
  \item \label{it:ffindim}
  For any $d \ge 0$, an element of $\sK^d(M)$ is uniquely determined by its $d$-th germ at any point of $M$. In particular, the space of $\sK^d(M)$ is finite dimensional. 

  \item \label{it:fgerm}
  Every space $\sK^d(M)$ is a $G$-module and a $\g$-module, where $G$ is the full isometry group of $M$ and $\g$ its Lie algebra.

  \item \label{it:fsym}
  If $K^1 \in \sK^{d_1}(M)$ and $K^2 \in \sK^{d_2}(M)$, then $K^1 \odot K^2 \in \sK^{d_1+d_2}(M)$, where $K^1 \odot K^2$ is the full symmetrisation of the tensor product $K^1 \otimes K^2$, so that the space $\sK(M)$ is a graded associative, commutative algebra relative to $\odot$.

  \item \label{it:fpoisson}
  If $K^1 \in \sK^{d_1}(M)$ and $K^2 \in \sK^{d_2}(M)$, then $\{K^1,K^2\} \in \sK^{d_1+d_2-1}(M)$, where $\{\cdot,\cdot\}$ is the Schouten bracket (defined by the extension of the Lie bracket to the tensor algebra; the function on $T^*M$ corresponding to $\{K_1,K_2\}$ is the Poisson bracket of the functions $\cK_1$ and $\cK_2$ corresponding to $K_1$ and $K_2$, respectively).

  \item \label{it:fanaly}
  If $(M,g)$ is analytic, then any element of $\sK(M)$ also is.

  \item \label{it:ftg}
  The restriction of a (covariant) Killing tensor field to a totally geodesic submanifold is Killing.
\end{enumerate}

For $d \ge 1$, denote $\sS^d(M)$ the linear span of the $d$-th symmetric products of the Killing vector fields, the space of decomposable elements of $\sK^d(M)$. Clearly, $\sS^d(M)$ is finite-dimensional $G$-module.

Clearly, the algebra of contravariant tensor fields equipped with the symmetric tensor product is isomorphic to the algebra of functions on $T^*M$ polynomial in momenta, as the graded, associative, commutative algebra. This isomorphism extends to the Lie algebra structures on them defined by the Schouten and the Poisson brackets, respectively. We will use the following notational convention: for a contravariant tensor field, we use a capital Roman letter (e.g., $K$), and for the corresponding function on $T^*M$, the same letter in calligraphic script (respectively, $\cK$).

In the following lemma we show that on a simply connected homogeneous space, a locally defined Killing tensor field can be extended to the whole space. As one can immediately see from an example of the flat torus, the condition of simply-connectedness is essential. Moreover, on a $C^\infty$ Riemannian space, one can have a local Killing vector field which does not extend globally. For the purposes of our study, the homogeneity assumption is sufficient, although one may ask if it can be relaxed to analyticity.

\begin{lemma} \label{l:ltog}
  Let $M$ be a simply connected homogeneous space and let $U \subset M$ be a connected, open subset. If $K' \in \sK^d(U)$ for some $d \ge 1$, then there exists a unique $K \in \sK^d(M)$ such that the restriction of $K$ to $U$ coincides with $K'$.
\end{lemma}
\begin{proof} 
  As a homogeneous space is analytic, any Killing tensor field on an open, connected subset of it is also analytic. Therefore without loss of generality, we can assume that $U$ is a small geodesic ball.

  The linear space of Killing tensor fields of rank $d$ on a Riemannian manifold is linearly isomorphic to the space of parallel sections of a certain linear connection $\nabla^d$ on the $d+1$-st jet bundle $J^{d+1}$ on the manifold (this fact is implicitly contained in \cite[Section~4]{WT}, and explicitly, in \cite[Theorem~8]{KM}). Let $N$ be the dimension of that subspace for Killing tensors of rank $d$ over $U$. As $M$ is homogeneous, this dimension is the same for all points of $M$, and moreover, by the induced action of the isometry group of $M$ on $J^{d+1}$, we obtain an $N$-dimensional sub-bundle $\Delta$ of $J^{d+1}(M)$ spanned by locally parallel (and locally defined) sections of $\nabla^d$. But $\Delta$ is $\nabla^d$-invariant, and the restriction of $\nabla^d$ to $\Delta$ is flat. As $M$ is simply connected, the holonomy group of this restriction is trivial, which implies that any local parallel section extends to a global one.
\end{proof}

\subsection{Symmetric spaces} 
\label{ss:symsp}

General references for theory of symmetric spaces are~\cite{H,WJ}.

Let $(M,g)$ be a simply connected Riemannian globally symmetric space of dimension $n \ge 3$, with the presentation $M=G/H$, where $G$ is the connected component of the full isometry group of $M$, $H$ is the isotropy subgroup at the point $o \in M$, which is the coset space $eH$, where $e \in G$ is the identity element. We denote $\g$ and $\h$ the Lie algebras of $G$ and $H$ respectively, and we identify $T_oM$ with the $\h$-module $\m \subset \g$ orthogonal to $\h$ relative to the Killing form of $\g$. 

We have $R(X,Y)Z=-[[X,Y],Z]$ for $X,Y,Z \in \m$. In particular, the Jacobi operator is given by $R_X Y = -\ad_X^2 Y$ for $X,Y \in \m$, where $\ad_T T' = [T,T']$ for $T,T' \in \g$.

\section{Killing tensor fields on symmetric spaces in normal coordinates.}
\label{s:kts}

Let $(M^n,g)$ be a Riemannian manifold. We introduce geodesic normal coordinates $(x^1, \dots, x^n)$ centered at a point $o \in M$ and defined on a neighbourhood $U$ of $o$, and denote $(p_1, \dots, p_n)$ the corresponding coordinates on the fibers of cotangent bundle to $M$. Under the pullback of the map $\exp_o:T_oM \to M$, this gives us the coordinates on the fibers of $T^*T_oM$. To an element of this bundle we associate a pair of vector $(X, P) \in T_oM \oplus T_oM$, where $X= (x^1, \dots, x^n)$ and $P=(p^1, \dots, p^n)$, where $p^i = p_i$, so that $P$ is the vector dual to the corresponding element of $T_X^*T_oM$ relative to the inner product $\ip$ on $T_oM$.

Local geometry of a symmetric space is fully encoded in its curvature tensor $R$ at a single point. We take $o \in M$ to be the projection of the identity in the presentation $M=G/H$ as in Subsection~\ref{ss:symsp}, and denote $\m = T_oM$. For $X \in \m$, we denote $R_X$ the Jacobi operator on $\m$ acting by $R_XY = R(Y,X)X$.

\begin{proposition} \label{p:Jac} 
  Let $(M,g)$ be a locally symmetric space. In the above notation, the following holds.
  \begin{enumerate}[label=\emph{(\alph*)},ref=\alph*]
    \item \label{it:JacgH}
    The metric tensor and the Hamiltonian are given, respectively, by
    \begin{align}
     \begin{split} \label{eq:symnormal}
        g(X) & = (t^{-1}\sin^2 \sqrt{t})|_{t=R_X} = (u^{-2}\sinh^2 u)|_{u=\ad_X} \\
            & = \sum_{m=0}^{\infty} \frac{2^{2m+1}(-1)^m}{(2m+2)!}R_X^m=\id - \frac13 R_X + \frac{2}{45}R_X^2 + O(\|X\|^6).
     \end{split}
     \\
     \begin{split} \label{eq:Hsym}
     \cH(X,P) & =\frac12 g^{ij}(x)p_ip_j = \Big\< \frac{t}{2\sin^2(\sqrt{t})}\Big|_{t=R_X} P, P\Big\>  = \sum_{m=0}^{\infty} c_m \<R_X^m P,P\> \\
       & = \frac12 \|P\|^2 + \frac16 R(P,X,X,P)+\frac{1}{30}\<R_X^2 P,P\>+ O(\|X\|^6),
     \end{split}
    \end{align}
    for $X, P \in \m$, where for $m \ge 0$,
    \begin{equation}\label{eq:cm}
      c_m = \frac{(-1)^{m+1}(2m-1)2^{2m-1}B_{2m}}{(2m)!},
    \end{equation}
    where $B_{2m}$ is the $(2m)$-th Bernoulli number.

    \item \label{it:JacV}
    For the space $\sK^{1}$ of Killing vector fields we have a decomposition $\sK^{1} = \sK^{0,1} \oplus \sK^{1,1}$, where the subspaces $\sK^{0,1}$ and $\sK^{1,1}$ are spanned respectively by the vector fields $V^0_A$ and $V^1_v$ defined by
    \begin{align} \label{eq:veven}
        \<V^0_A(X),P\> & = \<[A,X],P\>, \qquad A \in \h, \qquad \text{and} \\
     \begin{split} \label{eq:vodd}
     \<V^1_v(X),P\> & =\<v, (\sqrt{t} \cot\sqrt{t})|_{t=R_X} P\> = \<v, (u \coth u)|_{u=\ad_X} P\>  \\
       & = \<v, P-\frac13 R_X P - \frac1{45} R_X^2 P\> + O(\|X\|^6), \qquad v \in \m.
     \end{split}
    \end{align}
  \end{enumerate}
\end{proposition}

For the Hamiltonian $\cH$ given by~\eqref{eq:Hsym}, relative to the orthonormal basis $\{e_i\}$ for $\m$ such that $X=\sum_i x^i e_i$ and $P= \sum_i p_i e_i$, one has
  \begin{equation}\label{eq:dH}
  \begin{split}
    \frac{\partial}{\partial x^i} \cH(X, P) 
    &= 2 \sum_{a,b=0}^{\infty} c_{a+b+1} R(X,R_X^a P,R_X^b P,e_i), \\
    \frac{\partial}{\partial p_i} \cH(X, P) &= 2 \sum_{m=1}^{\infty} c_m \<R_X^m P,e_i\>.
  \end{split}
  \end{equation}
\begin{proof}
  For assertion~\eqref{it:JacgH}, take a nonzero $X \in \m$. Under $\exp_o$, the line $\gamma(s)=sX$ is a geodesic of $(M,g)$, and the vector field $J(s)= s P^i \partial/\partial x^i$ is a Jacobi field satisfying $J(0)=0,\, \dot{J}(0)=P$, where dot is the derivative relative to the affine parameter $s$. We have $\ddot{J}=-R_XJ$, and as $R_X$ is a constant matrix along $\gamma(s)$, we find $J(t)=(\frac{\sin (s\sqrt{t})}{\sqrt{t}})|_{t=R_X}P$. It follows that $g_{ij}(sX)P^iP^j = s^{-2} \|J(s)\|^2 = \<(\frac{\sin^2 (s\sqrt{t})}{s^2t})|_{t=R_X}P, P\> = \<(\frac{\sin^2 (\sqrt{t})}{t})|_{t=R_{sX}}P, P\>$, which proves the first equation of~\eqref{eq:symnormal}. Then the remaining equations of~\eqref{eq:symnormal} and~\eqref{eq:Hsym} easily follow.

  For assertion~\eqref{it:JacV} one verifies using~\eqref{eq:dH} that the functions $\cV^0_A=\<V^0_A(X),P\>$ and $\cV^1_v=\<V^1_v(X),P\>$ satisfy the equations $\{\cH,\cV^0\}=\{\cH,\cV^1_v\}=0$ and then notices that $\dim \sK^{0,1} + \dim \sK^{1,1} = \dim \h + \dim \m = \dim \g$ and that $\sK^{0,1} \cap \sK^{1,1} = 0$.
\end{proof}

For symmetric spaces of rank one we obtain a closed form for the Hamiltonian, as below.

\begin{corollary} \label{c:rk1}
  Let $(M,g)$ be a compact Riemannian symmetric space of rank one of non-constant curvature, with the metric scaled in such a way that the sectional curvature lies in $[1,4]$.

  In the notation of Proposition~\ref{p:Jac}, the Hamiltonian is given by
  \begin{equation}\label{eq:Hrk1}
     \cH(X,P) =\tfrac12 \|P\|^2 - \phi(\|X\|^2) (\|X\|^2\|P\|^2-\<X,P\>^2) + \tfrac{1}{6\cos^2 \|X\| } R(P,X,X,P),
  \end{equation}
  for $X, P \in \m$, where $\phi(\|X\|^2)= \frac{3 \cos^2\|X\|(\sin^2\|X\|-\|X\|^2)+\|X\|^2\sin^2\|X\|}{6\|X\|^2\cos^2\|X\|\sin^2\|X\|}$. In particular, for the spaces $\bc P^m$ and $\bh P^m$, we have
  \begin{align}\label{eq:Hcpm}
    \cH_{\bc P^m}(X,P) & =\tfrac12 \|P\|^2 + \psi(\|X\|^2) (\|X\|^2\|P\|^2-\<X,P\>^2) + \tfrac{1}{2\cos^2 \|X\| } \<JX,P\>^2, \\
    \cH_{\bh P^m}(X,P) & =\tfrac12 \|P\|^2 + \psi(\|X\|^2) (\|X\|^2\|P\|^2-\<X,P\>^2) + \tfrac{1}{2\cos^2 \|X\| } \sum_{\a=1}^{3}\<J_\a X,P\>^2, \label{eq:Hhpm}
  \end{align}
  respectively, where $J$ is the complex structure on $\m=T_o\bc P^m$ and $\Span(J_1,J_2,J_3)$ is the quaternionic structure on $\m=T_o\bh P^m$, and $\psi(\|X\|^2)= \tfrac{\|X\|^2-\sin^2\|X\|}{2\|X\|^2\sin^2\|X\|}$.
\end{corollary}
Note that both functions $\phi$ and $\psi$ are analytic, with a removable singularity at zero. 
\begin{proof}
  The curvature tensor of a rank one symmetric space of non-constant curvature has the following Osserman property: the restriction of the Jacobi operator $R_X$ to the subspace $X^\perp$ has two eigenvalues, $\|X\|^2$ and $4\|X\|^2$. It follows that the operator $\frac{t}{2\sin^2(\sqrt{t})}|_{t=R_X}$ from the right-hand side of the formula~\eqref{eq:Hsym}, when restricted to $X^\perp$, has the eigenvalues $\frac{\|X\|^2}{2\sin^2(\|X\|)}$ and $\frac{4\|X\|^2}{2\sin^2(2\|X\|)}$ respectively, with the same eigenspaces as $R_X$, and hence can be expressed as a linear combination of the restrictions of $R_X$ and $\id$ to $X^\perp$ with coefficients being certain functions of $\|X\|$. Computing these functions and then using the fact that $\frac{t}{2\sin^2(\sqrt{t})}|_{t=R_X}(X)=\frac{1}{2}X$ we get~\eqref{eq:Hrk1}. Equations~\eqref{eq:Hcpm} and~\eqref{eq:Hhpm} easily follow as the curvature tensors of $\bc P^m$ and $\bh P^m$, with our scaling, are given by $R(P,X,X,P)=(\|X\|^2\|P\|^2-\<X,P\>^2) + 3 \<J X,P\>^2$ and $R(P,X,X,P)=(\|X\|^2\|P\|^2-\<X,P\>^2) + 3 \sum_{\a=1}^{3}\<J_\a X,P\>^2$, respectively.
\end{proof}

Recall that throughout the paper, we denote a function $T^*M$ corresponding to a contravariant tensor field (say $K$) by the same letter in calligraphic script ($\cK$). Moreover, when working in geodesic normal coordinates, we write $\cK(X,P)$, where $X \in \m$ corresponds to the point of $M$ and $P \in \m$ is the \emph{vector} dual to the corresponding covector, \emph{relative to the inner product $\ip$} on $T_oM$.

Let $(M,g)$ be a connected, simply connected globally symmetric space. For $y \in M$, the geodesic reflection $r_y:M \to M$ at the point $y$ is an isometry. Its induced action $r_y^*$ on the tensor space is an involution on each of $\sK^d(M)$, and relative to geodesic normal coordinates centred at $y$, we have $(r_y^*\cK)(X,P)=\cK(-X,-P)$. We call the $(+1)$-eigenspace (respectively, the $(-1)$-eigenspace) of $r_y^*$ the space of \emph{even} (respectively, \emph{odd}) Killing tenor fields at $y$. It is clear that, relative to any point $y \in M$, a Killing tensor field splits into the sum of an even and an odd Killing tensor fields: $\cK=\frac12 (\cK + r_y^*\cK) + \frac12 (\cK - r_y^*\cK)$.

At any point, this splitting defines the structure of a bi-graded associative, commutative algebra on $\sK(M)$, with the bi-gradation $\sK^{a,d}(M), \; (a,d) \in \mathbb{Z}_2 \times \mathbb{Z}_{\ge 0}$ defined by requiring $\sK^{a,d}(M)$ to be the space of even (respectively, odd) elements of $\sK^d(M)$ when $a=d \mod 2$ (respectively, $a=(d+1) \bmod 2$). One clearly has $\sK^{a_1,d_1} \odot \sK^{a_2,d_2} \subset \sK^{a_1+a_2,d_1+d_2}$.

Given an even (respectively, an odd) contravariant tensor field $K$, we have the Taylor expansion with $b=0$ (respectively, with $b = 1$) relative to geodesic normal coordinates, given by
\begin{equation} \label{eq:taylorK}
\cK(X,P)=\sum_{s=0}^{\infty} K_{b+2s}(X^{b+2s},P^d) \in \sK^{(d+b) \bmod 2,d},
\end{equation}
where we abbreviate a sequence of $k \ge 0$ copies of an element $Y \in \m$ to $Y^k$. Note that for every $s \ge 0$, the term $K_{b+2s}$ in this expansion is a constant tensor of type $(0, b+2s+d)$ on $\m$ which is symmetric in the first $b+2s$ arguments and in the last $d$ arguments. 
Computing the Poisson bracket $\{\cH,\cK\}$ for $\cH$ given in~\eqref{eq:dH} and rearranging the terms we obtain that the tensor field $K$ is Killing if and only the following equations are satisfied for all $s \ge 0$:
\begin{align}
\begin{gathered}\label{eq:bracketHKe}
(2s+2)K_{2s+2}(X^{2s+1},P^{d+1})
  = 2\sum_{m=0}^{s} c_{s+1-m} \Big(- 2m K_{2m}(X^{2m-1},R_X^{s+1-m} P, P^d) \\
  + d K_{2m}\big(X^{2m}, \sum_{k=0}^{s-m} R(X,R_X^k P) R_X^{s-m-k}P ,P^{d-1}\big) \Big), \quad \text{for } b=0,
\end{gathered} \\
\begin{gathered}\label{eq:bracketHKo}
(2s+1)K_{2s+1}(X^{2s},P^{d+1})
  = 2\sum_{m=0}^{s-1} c_{s-m} \Big(- (2m+1) K_{2m+1}(X^{2m},R_X^{s-m} P, P^d) \\
  + d K_{2m+1}\big(X^{2m+1}, \sum_{k=0}^{s-m-1} R(X,R_X^k P) R_X^{s-m-k-1}P ,P^{d-1}\big) \Big), \quad \text{for } b=1,
\end{gathered}
\end{align}
where the coefficients $c_m$, are given by~\eqref{eq:cm}. Equations~\eqref{eq:bracketHKe} and~\eqref{eq:bracketHKo} can be written as a single equation: for all $N \ge 1$, we have:
\begin{equation}\label{eq:bracketHKP}
  \{\tfrac12 \|P\|^2, K_{N}(X^{N},P^d)\}=-\sum_{m=1}^{\lfloor N/2\rfloor} c_m\{\<R_X^m P, P\>, K_{N-2m}(X^{N-2m},P^d)\}.
\end{equation}

One immediate consequence of the equations~\eqref{eq:bracketHKe} and~\eqref{eq:bracketHKo} is as follows. By Lemma~\ref{l:young}, a (constant) tensor $K_N$ of type $(0,N+d)$ symmetric by the first $N$ and by the last $d$ arguments is uniquely determined by its values $K_N(X^{N-1},P^{d+1})$ for $X, P \in \m$, provided $N > d$. It follows that if all the terms $K_N, \, N \le d$, for a Killing tensor field $K$ of rank $d$ are zeros, then $K=0$ (in accordance with Fact~\ref{it:ffindim}).

\begin{lemma} \label{l:young}
  For integers $a > b > 0$, denote $S_{a,b}$ the map on the space of tensors of type $(0,a+b)$ which symmetrizes a given tensor $T$ by the first $a$ and the last $b$ arguments. If for $a',b'$ with $a'+b'=a+b$ and $b \le a' \le a$ we have $S_{a',b'}S_{a,b}T = 0$ for some $T$, then $S_{a,b}T = 0$.
\end{lemma}
\begin{proof}
  The image of the map $S_{a,b}$ is the tensor product of the spaces of symmetric tensors of type $(0,a)$ and $(0,b)$. By the Pieri's formula, this product decomposes into the sum direct sum of the irreducible modules of the action of the permutation group $\cS_{a+b}$ corresponding to the Young tableaux consisting of no more than two rows, with the first row being of size at least $a$ \cite[Section~I.4, Appendix~A]{FH}. Then the action of $S_{a',b'}$, when $b \le a' \le a$ is injective on each of these modules, and the claim follows.
\end{proof}

\section{Structural results. Proof of Theorem~\ref{t:prod} and Theorem~\ref{t:dual}}
\label{s:structure}

\subsection{Duality}
\label{ss:dual}

Recall that two simply connected irreducible globally symmetric spaces $(M,g)$ and $(M',g')$ are called \emph{dual}, if there is a linear isometry $\iota: T_o M \to T_{o'}M'$ (where $o \in M$ and $o' \in M'$ are arbitrary) such that for any $X,Y,Z \in T_oM$ one has $R'(\iota X,\iota Y)\iota Z = - R(X,Y)Z$, where $R$ and $R'$ are the curvature tensors of the spaces $(M,g)$ and $(M',g')$, respectively. It is known that in a pair of dual spaces, one is of compact type and another one is of non-compact type.

We extend this definition to arbitrary connected, simply connected, globally symmetric spaces, not necessarily irreducible, in the obvious way. Namely, given the de Rham decomposition of such a space $(M,g)=(M_0,g_0) \times (M_1,g_1) \times \dots \times (M_m,g_m)$, where $(M_0,g_0)$ is Euclidean and $(M_i,g_i),\, i=1, \dots, m$, are irreducible, we define the dual connected, simply connected, globally symmetric space $(M',g')$ by setting $(M',g')=(M_0,g_0) \times (M_1',g_1') \times \dots \times (M_p',g_m')$. Note that for $o \in M$ and $o' \in M'$ there still exists a linear isometry $\iota: T_o M \to T_{o'}M'$ such that $\iota^* R = - R'$, for the corresponding curvature tensors.

\begin{proof}[Proof of Theorem~\ref{t:dual}]
  For $o \in M$ and $o' \in M'$, we identify the tangent spaces $T_oM$ and $T_{o'}M'$ via a linear isometry $\iota$ such that $\iota^* R = - R'$. We choose small neighbourhoods $U(o) \subset M$ and $U'(o') \subset M'$ on which both $\exp_o$ and $\exp_{o'}$ are injective and such that $U'(o') = \exp_{o'} \circ \iota \circ \exp_{o'}^{-1}U(o)$ and introduce the geodesic normal coordinates on $M$ centred at $o$ and on $M'$ centred at $o'$. As $R'_X=-R_X$ under our identification, equation~\eqref{eq:Hsym} tells us that the Hamiltonians $\cH$ and $\cH'$ of $M$ and of $M'$, respectively, relative to the corresponding geodesic normal coordinates, are related by $\cH'(X,P)=\cH(\ir X,P)$. It follows that $K = K(X) \in \sK(U) \otimes \bc$ if and only if over $\bc$ we have $\{\cH(X,P),\cK(X,P)\}=0$ if and only if $\{\cH(\ir X,P),\cK(\ir X,P)\}=0$ if and only if $\{\cH'(X,P),\cK(\ir X,P)\}=0$ if and only if $K(\ir X) \in \sK(U') \otimes \bc$.

  The map $K(X) \mapsto K'(X)=K(\ir X,P)$ is an isomorphism of the complex, associative, commutative, graded algebras $\sK(U) \otimes \bc$ and $\sK(U') \otimes \bc$, and it extends to the isomorphism of $\sK(M) \otimes \bc$ and $\sK(M') \otimes \bc$, by Lemma~\ref{l:ltog} (note that the statement of the Lemma does not depend on whether the ground field is $\br$ or $\bc$ as the equations~\eqref{eq:bracketHKe} and~\eqref{eq:bracketHKo} defining the property of a tensor field to be Killing have real coefficients).
\end{proof}

\begin{remark} \label{rem:dualR}
  The isomorphism between the complex algebras $\sK(M) \otimes \bc$ and $\sK(M') \otimes \bc$ constructed in the proof does not descend to an isomorphism between the real algebras $\sK(M)$ and $\sK(M')$. Instead, one has linear space isomorphisms given as follows. An element $K \in \sK^{d \bmod 2,d}(M)$ corresponding to the function $\cK(X,P)=\sum_{s=0}^{\infty} K_{2s}(X^{2s},P^d)$ is mapped to the element $K' \in \sK^{d \bmod 2,d}(M')$ corresponding to the function $\cK'(X,P)=\sum_{s=0}^{\infty} (-1)^s K_{2s}(X^{2s},P^d)$, and similarly, an element $K \in \sK^{(d+1) \bmod 2,d}(M)$ with the function $\cK(X,P)=\sum_{s=0}^{\infty} K_{2s+1}(X^{2s+1},P^d)$, to the element $K' \in \sK^{(d+1) \bmod 2,d}(M')$ with the function $\cK'(X,P)=\sum_{s=0}^{\infty} (-1)^s K_{2s+1}(X^{2s+1},P^d)$.

  Clearly, these maps preserve the dimensions of the homogeneous components and the property of being decomposable.
\end{remark}

\subsection{De Rham decomposition}
\label{ss:deRham}

As it is shown in~\cite[Theorem~1.1]{MN2}, the algebra of Killing tensor fields on the product of two connected Riemannian spaces, at least one of which is compact, is the symmetric product of the algebras of Killing tensor fields on the factors (with no assumptions, an analogous statement is false even for spaces without flat factors).

\begin{proof}[Proof of Theorem~\ref{t:prod}]
  By~\cite[Theorem~1.1]{MN2}, it suffices to establish the claim in the case when all the factors are noncompact (either Euclidean or irreducible of noncompact type) and there are at least two of them. But then using Theorem~\ref{t:dual} and the explicit formulas given in Remark~\ref{rem:dualR} we can pass to the dual $M'$ of $M$ which is the product of compact irreducible symmetric spaces and possibly a Euclidean factor, apply~\cite[Theorem~1.1]{MN2} to $M'$, and then pass back to $M$ using the formulas in Remark~\ref{rem:dualR} again.
\end{proof}

As a consequence, we have the following useful fact.

\begin{corollary} \label{c:torus}
  Let $(M,g)$ be a connected, simply connected, globally symmetric space without flat factors, and let $X, P \in \m (\, = T_oM)$ be such that $[X, P]=0$ (equivalently, $R(X,P)=0$). If for $K \in \sK(M)$ we have $\cK(0,P) = 0$, then
  $\cK(X,P)=0$, in the notation of Proposition~\ref{p:Jac}.
\end{corollary}
\begin{proof}
  If $M$ is reducible, then the condition $R(X,P)=0$ is equivalent to the same condition for the projections of $X$ and $P$ to every factor, and so by Theorem~\ref{t:prod} it is sufficient to establish the claim assuming $M$ is irreducible. Moreover, if $M$ is irreducible and is of noncompact type, then the condition $R(X,P)=0$ is equivalent to $R'(\iota^* X, \iota^* P)=0$, where $R'$ is the curvature tensor of the dual space $M'$ and $\iota$ is the corresponding linear isomorphism. By Theorem~\ref{t:dual} (or by the explicit formulas given in Remark~\ref{rem:dualR}), we can assume $M$ to be of compact type.

  As $X, P \in \m$ commute, they lie in some Cartan subalgebra $\t \subset \m$. The subalgebra $\t$ is tangent to a totally geodesic flat torus $F \subset M$ of dimension $\rk M$. The restriction of the Killing tensor field $K$ is a Killing tensor field on $F$ (Fact~\ref{it:ftg}). But any Killing tensor field on a flat torus is parallel, and the claim follows.
\end{proof}

\section{Top slot Killing tensor fields. Proof of Theorem~\ref{t:quadratic}}
\label{s:top}

Although~\eqref{eq:bracketHKe} and~\eqref{eq:bracketHKo} are finite systems of linear equations in a finite dimensional space, they are very difficult to analyse. We single out a subspace of Killing tensor fields for which the these equations take a nice and elegant form.

Let $(M,g)$ be a connected, simply connected, globally symmetric space. Recall that a Killing tensor field $K \in \sK^d(M)$ is said to be \emph{top slot at a point $y \in M$}, if the Taylor decomposition of $\cK$ at $y$ contains no terms of degree less than $d$ in $X$. For a top slot tensor field, the Taylor expansion given by~\eqref{eq:taylorK} in geodesic normal coordinates has the form
\begin{equation} \label{eq:taylorKts}
\cK(X,P)=\sum_{s=0}^{\infty} K_{d+2s}(X^{d+2s},P^d),
\end{equation}

We denote $\sL^d \subset \sK^d(M)$ the space of top slot Killing tensor fields of rank $d$ at the point $o \in M$, the projection of the identity element of $G$ in the presentation $M=G/H$. Note that the space $\sL^d$ is an $H$-module and that it is nontrivial: for example, it contains all the symmetrised products of $d$ Killing vector fields vanishing at $o$ (such Killing fields are given by~\eqref{eq:veven}).

We prove that the Taylor expansion of the top slot tensor field consists of a single homogeneous term, so that in formula~\eqref{eq:taylorKts}, $K_{d+2s}=0$ when $s>0$, and that the integrability condition for that term take especially nice, explicit form. 

\begin{proposition} \label{p:top}
  Let $(M,g)$ be a connected, simply connected, globally symmetric space. Then the following holds in the notation of Proposition~\ref{p:Jac}, relative to geodesic normal coordinates centered at $o$.
  \begin{enumerate}[label=\emph{(\alph*)},ref=\alph*]
    \item \label{it:top1}
    For any $K \in \sL^d$, the Taylor expansion at $o$ of the function $\cK$ is given by $\cK(X,P)=K_d(X^d,P^d)$, where $K_d(X^d,P^d)$ is a constant (covariant) tensor on $\m = T_oM$, symmetric in the first $d$ and in the last $d$ arguments.

    \item \label{it:topint}
    Conversely, let $K$ be a contravariant tensor field of rank $d$ on $M$ such that the Taylor expansion at $o$ of the function $\cK$ is given by $\cK(X,P)=K_d(X^d,P^d)$, where $K_d$ is a constant tensor on $\m$ which is symmetric in the first $d$ and in the last $d$ arguments.

    Then $K$ is Killing if and only if the following $d+1$ linear equations are satisfied: for all $s =0, \dots, d$ and all $X, P \in \m$,
    \begin{equation} \label{eq:topint}
      d \, K_d(X^{d-1},P^{d-s+1}, (R_X P)^s) + s K_d(X^d,P^{d-s}, (R_X P)^{s-1},R(X,P)P) = 0,
    \end{equation}
  \end{enumerate}
\end{proposition}

We emphasize that the fact that the tensor $K_d$ on $\m$ is covariant, as opposed to the fact that the tensor field $K$ is contravariant, should not be confusing; we lift and lower the indices in the tensor algebra over $\m$ using the Euclidean inner product $\ip$ on $\m$, not the metric $g(X)$ given by~\eqref{eq:symnormal}.

\begin{proof}
  For assertion~\eqref{it:top1}, take $K \in \sL^d$. Seeking a contradiction, suppose that $q > 0$ is the smallest number for which the term $K_{d+2q}$ of the Taylor expansion~\eqref{eq:taylorKts} is nonzero, so that $\cK(X,P)=K_d(X^d,P^d) + K_{d+2q}(X^{d+2q},P^d) + \sum_{s=q+1}^{\infty} K_{d+2s}(X^{d+2s},P^d)$, where we can also assume that $K_d \ne 0$, as otherwise $K=0$ by Fact~\ref{it:fgerm}. From~\eqref{eq:bracketHKP} with $N=d$ and $N=d+2q$ we obtain, respectively,
  \begin{gather}\label{eq:PoisKd}
    \{\tfrac12 \|P\|^2, K_d(X^d,P^d)\}= 0,\\
    \{\tfrac12 \|P\|^2, K_{d+2q}(X^{d+2q},P^d)\}=- c_q\{\<R_X^q P, P\>, K_d(X^d,P^d)\}. \label{eq:PoisKd2q}
  \end{gather}
  Act on both sides of equation~\eqref{eq:PoisKd2q} by the operator $\mathcal{A}^{2q-1}$, where $\mathcal{A}=\{\tfrac12 \|P\|^2, \cdot \}$. By~\eqref{eq:PoisKd} and the Jacobi identity we obtain that the right-hand side of the resulting equation equals $- c_q\{\mathcal{A}^{2q-1}\<R_X^q P, P\>,K_d(X^d,P^d)\}$. But $\mathcal{A}^{2q-1}\<R_X^q P, P\>=0$. Indeed, the action of the operator $\mathcal{A}$ on a polynomial function of $X$ and $P$ has an effect of replacing separately every single argument $X$ by $P$ and then taking the negative of the sum of the resulting expressions. Then $\mathcal{A}^{2q-1}\<R_X^q P, P\>= (2q-1)! \<R_P^q X, P\> = 0$. Thus we obtain $\mathcal{A}^{2q}(K_{d+2q}(X^{d+2q},P^d))=0$ which gives $K_{d+2q}(X^d,P^{d+2q})=0$. As the tensor $K_{d+2q}$ is symmetric by the first $d+2q$ arguments and by the last $d$ arguments, Lemma~\ref{l:young} implies that $K_{d+2q} = 0$, a contradiction.

  To prove assertion~\eqref{it:topint}, we use one of the equations~\eqref{eq:bracketHKe} or~\eqref{eq:bracketHKo} (depending on the parity of $d$). They give that $K$ is Killing if and only for all $N \ge 0$, one has $-K_d(X^{d-1},R_X^N P, P^d)+K_d(X^d,\sum_{k=0}^{N-1} R(X,R_X^k P) R_X^{N-1-k} P, P^{d-1})=0$. Multiplying this equation by $t^N$ and taking the sum by $N$ from $0$ to infinity we obtain
  \begin{equation*}
  -K_d(X^{d-1},(\id-tR_X)^{-1}P, P^d)+ t K_d(X^d, R(X,(\id-tR_X)^{-1} P) (\id-tR_X)^{-1}P, P^{d-1})=0
  \end{equation*}
  (for small $t \in \br$). Denoting $Q=(\id-tR_X)^{-1}P$ we get $-K_d(X^{d-1},Q, ((\id-tR_X)Q)^d)+ t K^d(X^d, R(X,Q) Q, ((\id-tR_X)Q)^{d-1})=0$. The terms of degree $s=0, \dots, d$ in $t$ in this polynomial equation give equations~\eqref{eq:topint}.
\end{proof}


For symmetric spaces of rank one, the conditions~\eqref{eq:topint} can be simplified even further, to just two equations.

\begin{corollary} \label{c:rk1top}
  Let $(M,g)$ be a compact, connected, simply connected Riemannian symmetric space of rank one of non-constant curvature with the metric scaled in such a way that the sectional curvature lies in $[1,4]$.

  Suppose that $K$ is a contravariant tensor field of rank $d$ on $M$ such that the Taylor expansion at $o$ of the function $\cK$ is given by $\cK(X,P)=K_d(X^d,P^d)$, where $K_d$ is a constant tensor on $\m$ symmetric in the first $d$ and in the last $d$ arguments.

  Then $K$ is Killing if and only if for all $X, P \in \m$, we have
    \begin{equation} \label{eq:toprk1}
      K_d(X^{d-1},P^{d+1}) = 0, \qquad K_d(X^{d-1},R(P,X)X, P^d) = K_d(X^d, P^{d-1},R(X,P)P),
    \end{equation}
  or equivalently, $\{\|P\|^2,K_d(X^d,P^d)\}=\{R(P,X,X,P),K_d(X^d,P^d)\}=0$.
\end{corollary}
\begin{proof}
  The condition $K_d(X^{d+1},P^{d-1}) = 0$ is necessary for the top slot tensor field $K$ to be Killing: this is the first equation of~\eqref{eq:topint}. This condition is equivalent to the fact the symmetrisation of $K_d$ by its first $d+1$ arguments is zero and is equivalent to the fact that $\{\|P\|^2,K_d(X^d,P^d)\}=0$. It also implies that $K_d(P^d,X^d) = (-1)^d K_d(X^d,P^d)$. This follows either from an argument similar to that in the proof of Lemma~\ref{l:young} (the only surviving module of the permutation group $\cS_{2d}$ is given by the Young tableau $(d,d)$), or directly, by noticing that for $a=d,d-1, \dots, 1,0$, one has $K_d(X^a,P^d,X^{d-a}) = - \frac{a}{d-a+1} K_d(X^{a-1},P^d,X^{d-a+1})$. But then swapping $X$ and $P$ we get $\{\|X\|^2,K_d(X^d,P^d)\}=0$, and hence $\{f(\|X\|^2),K_d(X^d,P^d)\}=0$, for any differentiable function $f(t)$. Using the fact that $\{\<X,P\>,K_d(X^d,P^d)\}=0$, as $K_d$ has the same degree in $X$ and in $P$, we obtain from~\eqref{eq:Hrk1} that the condition $\{\cH(X,P),K_d(X^d,P^d)\}=0$ is equivalent to $\{R(P,X,X,P),K_d(X^d,P^d)\}=0$ (which is equivalent to the second equation of~\eqref{eq:toprk1}), as required.
\end{proof}

\smallskip

We now restrict our attention to \emph{quadratic} Killing tensor fields and give the proof of Theorem~\ref{t:quadratic}. For quadratic fields, we can prove that top slot Killing tensor fields span the whole space of Killing tensor fields. It would be extremely interesting to know if this still holds for Killing tensor fields of higher rank. If this were true, the space of Killing tensor fields of any rank would be described by our equations~\eqref{eq:topint} in Proposition~\ref{p:top}.

\begin{proof}[Proof of Theorem~\ref{t:quadratic}]
  Assertion~\eqref{it:quadratictop1} is just the specification of Proposition~\ref{p:top}\eqref{it:top1} to the case $d=2$.

  Assertion~\eqref{it:quadraticeqs} also easily follows from Proposition~\ref{p:top}\eqref{it:topint} when $d=2$. Equation~\eqref{eq:topint} with $s=0$ gives $K_2(X,P,P,P) = 0$ which is equivalent to~\eqref{eq:quadraticsigma} by the symmetries of $K_2$. Equation~\eqref{eq:topint} with $s=1$ gives $2 K_2(X,P,P, R(P,X)X) + K_2(X,X,P,R(X,P)P) = 0$ which is equivalent to~\eqref{eq:quadratic21} modulo~\eqref{eq:quadraticsigma}. And finally, from equation~\eqref{eq:topint} with $s=2$ one obtains $K_2(X,P,R(P,X)X,R(P,X)X)+K_2(X,X,R(P,X)X,R(X,P)P)=0$. Replacing $X$ by $X+tP$ and taking the coefficient of $t^1$ in the resulting equation we obtain~\eqref{eq:quadratic22}. Conversely, replacing $P$ by $P+tX$ in~\eqref{eq:quadratic22} and taking the coefficient of $t^1$ we get back the former equation.

  For assertion~\eqref{it:quadraticspan}, denoting $\sW^2(M)=\Span(\cup_{x \in M} \sL^2_x) \subset \sK^2(M)$, we aim to prove that in fact, $\sW^2(M)=\sK^2(M)$.

  We claim that it suffices to prove that any element even at $o$ (the space of such elements is denoted $\sK^{0,2}_o(M)$ in the notation of Section~\ref{s:kts}) belongs to $\sW^2(M)$. Indeed, if this fact has already been established, we can further argue as follows. Let $G'$ be the subgroup of the full isometry group of $M$ generated by the reflections $r_x, \, x \in M$. Then the group $G$ either coincides with $G'$ (this happens when the ranks of $G$ and of $K$ are equal~\cite[Theorem~8.6.7]{WJ}) or is a subgroup of index $2$ in $G'$ (then $G'$ is disconnected). In any case, the group $G'$ is compact, and the space $\sK^2(M)$ is a finite dimensional $G'$-module. Note that $\sW^2(M)$ is also a $G'$-module, as it is generated (as a $G$-module) by the subspace $\sL^2_o$ which is invariant under the reflection at $o$. Take a nonzero element $K$ in an irreducible $G'$-submodule $\sK' \subset \sK^2(M)$. Acting by an element of $G$ if necessary we can assume that $K(o) \ne 0$. Then the element $K'=K + r_o^* K \in \sK'$ is still nonzero and is even at $o$. Then it belongs to $\sW^2(M)$. As $\sK'$ is irreducible, it is generated by any of its nonzero elements, in particular, by $K'$, and so $\sK' \subset \sW^2(M)$.

  It remains to show that the space $\sK^{0,2}_o(M)$ of quadratic Killing tensor fields that are even at $o$ lies in $\sW^2(M)$. Subtracting from any element of $\sK^{0,2}_o(M)$ a quadratic form in the Killing vector fields $V^1_v,\, v \in \m$, given by~\eqref{eq:vodd} which are odd at $o$ we obtain an element of $\sL^2_o$ which lies in $\sW^2(M)$ by definition, and so it is sufficient to prove that $V^1_u \odot V^1_v \in \sW^2(M)$, for any $u, v \in \m$.

  Any element $V^0_A \odot V^0_A,\, A \in \h$, where $V^0_A$ is given by~\eqref{eq:veven} is already top slot at $o$. Moreover, we have $\{V^0_A, V^1_v\}=V^1_{[A,v]}$. As $\sW^2(M)$ is a $G$-module, it contains the span of all the elements $\{\{V^0_A \odot V^0_A, V^1_v\}, V^1_v\} = 2 V^1_{[A,v]} \odot V^1_{[A,v]} + 2 V^0_{A} \odot V^0_{[[A,v],v]}$. But $[[A,v],v] \in \h$, and so $V^0_{A} \odot V^0_{[[A,v],v]}$ is top slot at $o$. Hence $V^1_{[A,v]} \odot V^1_{[A,v]} \in \sW^2(M)$. We note that the image of the map $\Phi: \h \times \m \to \m$ given by $\Phi(A,v) \mapsto [A,v]$, has non-empty interior. Indeed, at a point $(A,v)$ where $\Phi^*$ has the maximal rank (the set of such points is open), the tangent space to the image of $\Phi$ equals $[\h,v] + [A,\m]$, and so the normal space is the set of elements $w \in \h$ such that $\<[\h,v], w\> = \<[A,\m],w\> = 0$, which implies $[v,w] = [A,w] = 0$. Taking $v$ to be regular (the set of regular elements of $\m$ is open), we obtain that $w$ lies in the Cartan subspace $\t \subset \m$ determined by $v$. If $w$ is nonzero, we choose a root $\alpha$ such that $\alpha(w) \ne 0$ and take a nonzero element $A$ in the corresponding root subspace $\h_\a \subset \m$. Then $[A,w] \ne 0$, and so $\Phi^*$ is locally onto at $(A,v)$. This property is clearly open, and so the image of $\Phi$ has a nonempty interior. In particular, there exists an open set of regular elements $u \in \m$ belonging to that image. For such elements, we have $V^1_u \odot V^1_u \in \sW^2(M)$. Now, for any $v,w \in \m$, we have $\{\{V^1_u \odot V^1_u, V^1_v\}, V^1_w\} \in \sW^2(M)$, and so $V^1_u \odot V^1_{[[u,v],w]} \in \sW^2(M)$. Let $L_u = \Span([[u,v],w] \, | \, v, w \in \m) \subset \m$. If $x \in L_u^\perp$ is nonzero, then $x \perp [\h,u]$, and so $[u,x] = 0$, and so $x$ lies in the Cartan subspace of $\m$ determined by $u$. But then we have $[[v,u],x] = 0$, for all $v \in \m$, which leads to a contradiction, if we take $u$ belonging to a root subspace $\m_\alpha$ such that $\alpha(x) \ne 0$. Hence $L_u  = \m$. This implies that $V^1_v \odot V^1_u \in \sW^2(M)$, for any $v \in \m$ and any $u$ from an open subset of $\m$, and hence, for any $u \in \m$, as required.
\end{proof}

Theorem~\ref{t:quadratic} effectively reduces the problem of finding quadratic Killing tensor fields on a given symmetric space to solving a system of linear equations~(\ref{eq:quadraticsigma}, \ref{eq:quadratic21}, \ref{eq:quadratic22}). We make two observations of computational nature.
\begin{remark} \label{rem:Bianchi}
  First, it is easy to see that an element $K \in \sL^2$ is decomposable if and only if it is decomposable relative to $H$, that is, if it is comes from a quadratic form on the Lie algebra $\h$ under the projection $\Sym^2(\h) \ni A \odot B \mapsto V^0_A \odot V^0_B$, where $V^0_A$ is given in~\eqref{eq:veven}.

  Second, we note that equation~\eqref{eq:quadraticsigma} implies that the cyclic sum of $K_2$ by any three arguments is zero, and that $K_2(P,X,Y,Z)=K_2(Y,Z,P,X)$, for all $X, Y, Z, P \in \m$ (see the proof of Corollary~\ref{c:rk1top} for a more general fact). Furthermore, the $\SL(n)$-module of tensors of type $(0,4)$ having such symmetries is isomorphic to the $\SL(n)$-module of algebraic curvature tensors of type $(0,4)$; it has the same Young tableau \raisebox{-1pt}{\scalebox{1.5}{$\boxplus$}} and the same dimension $\frac{1}{12}n^2(n^2-1)$.
\end{remark}

\section{Proof of Theorem~\ref{t:hpnop2}}
\label{s:hpnop2}

We start with quadratic Killing tensor fields on the quaternionic projective space $\bh P^m=\Sp(m+1)/(\Sp(m)\Sp(1))$. We equip the space $\br^{4m+4}$ with the quaternionic structure defined by three anticommuting Hermitian structures $J_1, J_2$ and $J_3=J_1J_2$. The denote $\spg(m+1)$ and $V_{m+1}$ the space of skew-symmetric and of symmetric matrices commuting with $\Span(J_1,J_2,J_3)$, respectively, and denote $W_{m+1}$ the subspace of matrices of trace zero in $V_{m+1}$. We construct the following two subspaces of constant tensors $T$ of type $(0,4)$ on $\br^{4m+4}$ which are symmetric by the first pair of arguments and by the second pair of arguments. For an element $A \odot B \in \Sym^2(\spg(m+1))$ we define
\begin{equation}\label{eq:hpmT1}
\begin{gathered}
T^1_{A,B} (X,X,P,P) = \<AX,P\>\<BX,P\>, \quad \text{for } X, P \in \br^{4m+4}, \\
\sT_m^1 = \Span(T^1_{A,B} \, | \, A, B \in \spg(m+1)).
\end{gathered}
\end{equation}
For an element $S \odot Q \in \Sym^2(V_{m+1})$ we define
\begin{equation}\label{eq:hpmT2}
\begin{gathered}
T^2_{S,Q}(X,X,P,P) = \sum_{\a=1}^{3} \<SJ_\a X, P\>\<QJ_\a X, P\>, \quad \text{for } X, P \in \br^{4m+4}, \\
\sT_m^2 = \Span(T^2_{S,Q} \, | \, S, Q \in V_{m+1}).
\end{gathered}
\end{equation}
Put $\sT_m = \sT_m^1 + \sT_m^2$ (note that the map from $\Sym^2(\spg(m+1)) \oplus \Sym^2(V_{m+1})$ to $\sT_m$ always has a nontrivial kernel and that $\sT_m^1 \cap \sT_m^2 \ne 0$, but this will not affect our arguments).

The following properties are easy to check.
  \begin{enumerate}[label=(\Roman*),ref=\Roman*]
  \item \label{it:hpmTgeod}
  Any $T \in \sT_m$ satisfies $T(X,X,X,P)=0$, and so on every geodesic $\gamma(s) = \cos s X + \sin s Y$ of the unit sphere $S^{4n+3}$, where $X, Y \in \br^{4n+4}$ are orthonormal, the function $T(\gamma(s), \gamma(s), \dot{\gamma}(s), \dot{\gamma}(s))$ is constant.

  \item \label{it:hpmTinv}
  Any $T \in \sT_m$ is $\Sp(1)$-invariant, that is, $T(e^J X,e^J X,e^J P,e^J P) = T(X,X,P,P)$, for any $J \in \Span(J_1,J_2,J_3)$.
\end{enumerate}

Let $\pi: S^{4m+3} \to \bh P^m$ be the Riemannian submersion along the fibers of the Hopf fibration defined by the orbits of the action of $\Sp(1)$ on $S^{4m+3}$.

\begin{lemma} \label{l:hpmprep}
  Let $T$ be a constant tensor of type $(0,4)$ on $\br^{4m+4}$ which is symmetric by the first pair of arguments and by the second pair of arguments and which satisfies properties~\eqref{it:hpmTgeod} and~\eqref{it:hpmTinv} (but is not necessarily an element of $\sT_m$). The following hold.
  \begin{enumerate}[label=\emph{(\alph*)},ref=\alph*]
    \item \label{it:hpmprepFT}
    The quadratic tensor field $F_T$ on $\bh P^m$ defined by
    \begin{equation} \label{eq:FT}
        F_T(\pi(X),\pi(X),(d\pi)_X P,(d\pi)_X P)=T(X,X,P,P),
    \end{equation}
    for a $X \in S^{4m+3}$ and $P \in \br^{4m+4}$ is well-defined and Killing.

    \item \label{it:hpmprepkerF}
    The kernel of the map $T \mapsto F_T$ defined by~\eqref{eq:FT} is $\{T=T^2(S, I_{4m+4}) \, | \, S \in V_{m+1}) \subset \sT^2_m$, as given by~\eqref{eq:hpmT2}.

    \item \label{it:hpmprepp1p2p3}
    Suppose the polynomials $f_1, f_2$ and $f_3$ in $8m+8$ variables $X=(x_1, \dots, x_{4m+4})^t$ and $P=(p_1, \dots, p_{4m+4})^t$ satisfy $\sum_{\a=1}^{3} f_\a(X,P) \<J_\a X, P\> = 0$. Then $f_\a(X,P) = h_\b(X,P)\<J_\gamma X,P\> - h_\gamma (X,P)\<J_\b X,P\>$, where $(\a,\b,\gamma)$ is a cyclic permutation of $(1,2,3)$, for some polynomials $h_1, h_2, h_3$.
  \end{enumerate}
\end{lemma}
From assertion~\eqref{it:hpmprepFT} it follows that the quadratic tensor fields $F_T$ are Killing for $T \in \sT_m$. Note that for $T \in \sT_m^1$, the corresponding tensor fields $F_T$ are decomposable.
\begin{proof}
  For a tensor $T$ as in the assumption of the lemma, property~\eqref{it:hpmTinv} implies that the quadratic tensor field $F_T$ on $\bh P^m$ is well-defined. Moreover, from property~\eqref{it:hpmTgeod}, the tensor field $F_T$ is constant along geodesics of $\bh P^m$, as the latter are the projections of the horizontal geodesics on $S^{4n+3}$, and hence $F_T$ is a quadratic Killing tensor on $\bh P^m$, as claimed by assertion~\eqref{it:hpmprepFT}.

  We next establish assertion~\eqref{it:hpmprepp1p2p3}. It is easy to see that it suffices to prove it assuming $X, P \in \bc^{4m+4}$. Let $\cI$ be the ideal in $\bc[X,P]$ generated by polynomials $\<J_2X,P\>$ and $\<J_3X,P\>$. A Zariski open subset of the variety $V(\cI)=\{(X,P) \in \bc^{8m+8} \, | \, \<J_2X,P\> = \<J_3X,P\> = 0\}$ can be given by a rational parametrisation (we can express $x_1$ and $x_2$ as rational functions of $x_3, \dots, x_{4m+4},p_1, \dots, p_{4m+4}$), and so by \cite[Proposition~4.5.6]{CLO}, $V(\cI)$ is irreducible. Hence the ideal $\cI$ is prime. We have $f_1(X,P) \<J_1 X, P\> = 0$ for $(X,P) \in V(\cI)$, and as $\<J_1 X, P\>$ is not zero on $V(\cI)$, we obtain that $f_1(X,P) = 0$ on $V(\cI)$. As $\cI$ is prime we obtain $f_1(X,P) = r_{12}(X,P)\<J_2X,P\> + r_{13}(X,P)\<J_3X,P\>$, for some $r_{12}, r_{13} \in \bc[X,P]$. By a similar argument, we get another two equations obtained by a cyclic permutations of the subscripts $(1,2,3)$. From the assumption we get $\sum_{\a < \b} (r_{\a\b}(X,P)+r_{\b\a}(X,P))\<J_\a X,P\>\<J_\b X,P\> = 0$. Then for $\{\a,\b,\gamma\} = \{1,2,3\}$, we obtain that $r_{\a\b}(X,P)+r_{\b\a}(X,P)$ is divisible by $\<J_\gamma X, P\>$, and moreover, $r_{\a\b}(X,P)+r_{\b\a}(X,P)= q_\gamma(X,P) \<J_\gamma X,P\>$, for some $q_\gamma \in \bc[X,P]$ with $\sum_{\gamma=1}^{3} q_\gamma = 0$. It follows that $(r_{\a\b}(X,P)+q_\a(X,P) \<J_\gamma X, P\>) + (r_{\b\a}(X,P)+q_\b(X,P) \<J_\gamma X, P\>) = 0$, for $\{\a,\b,\gamma\} = \{1,2,3\}$, and the claim of the assertion follows if we define $h_\gamma(X,P) = r_{\b\a}(X,P)+q_\b(X,P) \<J_\gamma X, P\> = -(r_{\a\b}(X,P)+q_\a(X,P) \<J_\gamma X, P\>)$, where $(\a,\b,\gamma)$ is a cyclic permutation of $(1,2,3)$.

  For assertion~\eqref{it:hpmprepkerF} we suppose that $F_T = 0$ for some $T$ satisfying the assumptions of the lemma. By~\eqref{eq:FT} this means that $T(X,X,P,P)=0$ when $X \in S^{4m+3}$ and $P$ is horizontal, and hence $T(X,X,P,P)=0$ for all $X,P \in \br^{4m+4}$ such that $\<J_1X, P\> = \<J_2X, P\> = \<J_3X, P\> = 0$. An argument similar to that in the previous paragraph shows that in $\bc^{8m+8}$, the variety defined by these three equations is irreducible, and so $T(X,X,P,P)= \sum_{\a=1}^{3} f_\a(X,P) \<J_\a X,P\>$, for some polynomials $f_\a \in \bc[X,P]$. As $T$ and $J_\a$ have real coefficients, we can choose $f_\a$ to be real for $X,P \in \br^{4m+4}$. Furthermore, comparing the degrees, we can choose them bilinear, so that $T(X,X,P,P)= \sum_{\a=1}^{3} \<N_\a X,P\> \<J_\a X,P\>$, for some matrices $N_\a$. From property~\eqref{it:hpmTgeod} we obtain $\sum_{\a=1}^{3} \<N_\a X,X\> \<J_\a X,P\> = 0$, from which $N_\a \in \so(4m+4)$. Moreover, in the expression for $T$, we can replace each $N_\a$ by $N_\a + \sum_{\b=1}^{3} Q_{\a\b} J_\b$, where $Q \in \so(3)$, and so by a choice of $Q$ we can assume that the matrix $c_{\a\b}=\Tr (N_\a J_\b)$ is symmetric. Now from property~\eqref{it:hpmTinv} we obtain $0=\frac{d}{dt}|_{s=0} T(e^{sJ_1} X,$ $e^{sJ_1} X,e^{sJ_1} P,e^{sJ_1} P)= \sum_{\a=1}^{3} (\<[N_\a,J_1] X,P\> \<J_\a X,P\> + \<N_\a X,P\> \<[J_\a,J_1] X,P\>)$, and so assertion~\eqref{it:hpmprepp1p2p3} implies that for some $\mu, \nu, \eta \in \br$ we have
  \begin{equation*}
    [N_1,J_1] = \mu J_2 - \nu J_3, \quad [N_2,J_1] + 2 N_3 = -\mu J_1 + \eta J_3, \quad [N_3,J_1] - 2N_2 = \nu J_1 - \eta J_2.
  \end{equation*}
  But now multiplying the first equation by $J_2$ and taking the trace we obtain $2c_{13}= -(4m+4) \mu$, while multiplying the second equation by $J_1$ and taking the trace gives $2c_{13}=(4m+4)\mu$. It follows that $c_{13} = \mu = 0$, and by a similar argument, $c_{12}=\nu = 0$. Moreover, multiplying the second equation by $J_3$ and taking the trace we get $2c_{33}-2c_{22}= -(4m+4) \eta$, while multiplying the third equation by $J_2$ and taking the trace gives $2c_{33}-2c_{22} = (4m+4)\eta$. We obtain $c_{22}=c_{33}$ and $\eta = 0$. By symmetry we obtain $[N_\a, J_\a] = 0$ and $[N_\a, J_\b] = 2 \ve_{\a\b} N_\gamma$, for $\{\a,\b,\gamma\} = \{1,2,3\}$, where $J_\a J_\b = \ve_{\a\b} J_\gamma$. In particular, from $[N_1,J_2] = 2N_3$ we get $2N_3J_3 = N_1J_1 - J_2N_1J_3$, and from $[N_1,J_3] = -2N_2$ it follows that $-2N_2J_2 = -N_1J_1 - J_3N_1J_2 = -N_1J_1 + J_2N_1J_3$, as $J_3N_1J_2 = -J_2J_1N_1J_2 = -J_2N_1J_1J_2 = -J_2N_1J_3$. Adding the two equations we obtain $N_3J_3 = N_2J_2$, and so by symmetry, there exists a matrix $S$ such that $N_\a = SJ_\a$, for $\a=1,2,3$. Now from $[N_\a,J_a] = 0$ it follows that $S$ commutes with all $J_\a$, and then $0 = N_\a + N_\a^t = SJ_\a - J_a S^t = J_\a (S-S^t)$. Hence $S$ is symmetric, and so $S \in V_{m+1}$. This gives $T=T^2(S,I_{4m+4})$, as required.
\end{proof}

We now prove by induction by $m$ that any quadratic Killing tensor field on $\bh P^m$ has the form $F_T$ given in~\eqref{eq:FT}, where $T \in \sT_m$. The space $\bh P^1$ is a sphere, and so any Killing tensor field on it is decomposable.

Suppose that $m > 1$ and that the claim has already been established for all spaces $\bh P^k$ with $k<m$. Choose a point $o \in \bh P^m$ and consider a top slot quadratic Killing tensor field at the point $o$. By Theorem~\ref{t:quadratic}, it is uniquely defined by a constant tensor $K_2$ of type $(0,4)$ on $\br^{4m} = T_o \bh P^m$ which is symmetric by the first two arguments, symmetric by the second two arguments, and satisfies equations~\eqref{eq:quadraticsigma} and~\eqref{eq:quadratic21}; the fact that we do not require equation~\eqref{eq:quadratic22} follows from Corollary~\ref{c:rk1top}. For $X, P \in T_o \bh P^m$, the curvature tensor of $\bh P^m$ is given by $R(X,P)P = \|P\|^2 X - \<X, P\>P + 3 \sum_{\a=1}^{3} \<J_\a P, X\> J_\a P$, and so using~\eqref{eq:quadraticsigma} we obtain from~\eqref{eq:quadratic21} that $\sum_{\a=1}^{3} (K_2(X,X,P,J_\a P)+K_2(X,J_\a X,P,P))\<J_\a X,P\> = 0$. By Lemma~\ref{l:hpmprep}\eqref{it:hpmprepp1p2p3}, there exist polynomials $f_1,f_2,f_3$ in $X,P$ such that $K_2(X,X,P,J_\a P)+K_2(X,J_\a X,P,P) = f_{\a+1}(X,P) \<J_{\a+2} X,P\> - f_{\a+2}(X,P) \<J_{\a+1} X,P\>$, where the subscripts are computed modulo $3$. Comparing the degrees, we see that there exist matrices $N_\a$, such that for $\a = 1,2,3$, we have
\begin{multline}\label{eq:hpmK2}
  K_2(X,X,P,J_\a P)+K_2(X,J_\a X,P,P) \\ = \<N_{\a+1}X,P\> \<J_{\a+2} X,P\> - \<N_{\a+2} X,P\> \<J_{\a+1} X,P\>,
\end{multline}
with the subscripts computed modulo $3$.

For a tensor $T$ of type $(0,4)$ on $\br^{4m}$ we define
\begin{equation} \label{eq:hpmJaT}
\begin{split}
   (J_\a.T)(X_1,X_2,X_3,X_4) & = \tfrac{d}{dt}|_{t=0} T(e^{tJ_\a} X_1, e^{tJ_\a} X_1, e^{tJ_\a} X_1, e^{tJ_\a} X_1) \\
     & = T(J_\a X_1,X_2,X_3,X_4) + T(X_1,J_\a X_2,X_3,X_4) \\
     &+ T(X_1,X_2,J_\a X_3,X_4) + T(X_1,X_2,X_3,J_\a X_4),
\end{split}
\end{equation}
and so~\eqref{eq:hpmK2} can be written as
\begin{equation}\label{eq:hpmJaK}
(J_\a.K_2)(X,X,P,P) = 2 (\<N_{\a+1}X,P\> \<J_{\a+2} X,P\> - \<N_{\a+2} X,P\> \<J_{\a+1} X,P\>).
\end{equation}
Since $(J_\a.K_2)(X,X,X,P) = 0$ by~\eqref{eq:hpmJaT} and~\eqref{eq:quadraticsigma}, we obtain from~\eqref{eq:hpmJaK} that $N_\a \in \so(4m)$. We also note that for any $3 \times 3$ symmetric matrix $(c_{\a\b})$, equations~\eqref{eq:hpmJaK} are satisfied with $N^{(1)}_\a = - \sum_{\b=1}^{3} c_{\a\b} J_\b$ and with $K_2 = K_2^{(1)}$ given by $K_2^{(1)}(X,X,P,P) = \sum_{\b,\gamma=1}^{3} c_{\b \gamma} \<J_\b X, P\> \<J_\gamma X, P\>$. So taking $c_{\a\b}=-\frac12 \Tr(N_\a J_b + N_\b J_\a)$ and replacing $N_\a$ by $N_\a - N^{(1)}_\a$ and $K_2$ by $K_2 - K_2^{(1)}$, we can assume that in equations~\eqref{eq:hpmJaK} we additionally have $\Tr(N_\a J_b + N_\b J_\a) = 0$.

Furthermore, from~\eqref{eq:hpmJaT} it is easy to see that $(J_1.J_2.T) - (J_2.J_1.T) = -2 J_3.T$, and so for $T=K_2$, equation~\eqref{eq:hpmJaK} gives $\<[N_3,J_1]X,P\>\<J_1 X,P\> + \<[N_3,J_2]X,P\>\<J_2 X,P\> - (\<[N_1,J_1]X,P\> + \<[N_2,J_2]X,P\>)\<J_3 X,P\> = 0$. Then from Lemma~\ref{l:hpmprep}\eqref{it:hpmprepp1p2p3} we obtain that for some $\eta, \tau, \sigma \in \br$ we have
\begin{equation*}
    [N_3,J_1] = \eta J_2 + \tau J_3, \quad [N_3,J_2] = -\eta J_1 + \sigma J_3, \quad [N_1,J_1] + [N_2,J_2] = \tau J_1 + \sigma J_2.
\end{equation*}
Multiplying each of these three equations by $J_\a, \, \a=1,2,3$, and taking the trace we obtain $\Tr(N_3 J_2) = \Tr(N_2 J_3) = 2m\tau, \, \Tr(N_3 J_1) = \Tr(N_1 J_3) = -2m\sigma, \, \Tr(N_3 J_3) = -2m\eta$ and $\Tr(N_1J_2-N_2J_1) = 0$. As from the previous paragraph we can assume that $\Tr(N_\a J_b + N_\b J_\a) = 0$, we deduce that $\eta=\tau=0$, and so the by the above equation and the ones obtained from it by permuting the subscripts $\{1,2,3\}$ we get $[N_\a,J_\b] = 0$, for all $\a,\b=1,2,3$. It follows that $N_1,N_2,N_3 \in \spg(m)$. Then equations~\eqref{eq:hpmJaK} are satisfied with $K_2 = K_2^{(2)}$ given by $K_2^{(2)}(X,X,P,P) = \sum_{\a=1}^{3} \<N_\a X, P\> \<J_\a X, P\>$. Replacing $K^2$ by $K_2 - K_2^{(2)}$, we can now assume that the equations~\eqref{eq:hpmJaK} have the form $(J_\a.K_2)(X,X,P,P) = 0$, for $\a=1,2,3$.

But then the definition~\eqref{eq:hpmJaT} implies that $K_2$ is $\Sp(1)$-invariant. So the tensor $K_2$ on $\br^{4m}$ has the properties~\eqref{it:hpmTgeod} and~\eqref{it:hpmTinv}, from which it follows that the quadratic tensor field $F_{K_2}$ on $\bh P^{m-1}$ defined by formula~\eqref{eq:FT} for the Hopf projection $\pi:S^{4m-1} \to \bh P^{m-1}$ is well-defined and Killing, by Lemma~\ref{l:hpmprep}\eqref{it:hpmprepFT}. By the induction hypothesis, $F_{K_2} = F_T$ for some $T \in \sT_{m-1}$, and so by Lemma~\ref{l:hpmprep}\eqref{it:hpmprepFT} we obtain $K_2 \in \sT_{m-1}$.

Thus for any tensor $K_2$ satisfying conditions of Theorem~\ref{t:quadratic}\eqref{it:quadraticeqs} we have $K_2 - K_2^{(1)} - K_2^{(2)} \in \sT_{m-1}$. We now show that there exists $T \in \sT_m$ for which the quadratic tensor field $F_T$ defined by~\eqref{eq:FT} is top slot at $o$ and such that the $(0,4)$ tensor on $T_o \bh P^m$ constructed as in Theorem~\ref{t:quadratic}\eqref{it:quadratictop1} from $F_T$ coincides with $K_2$.

Choose an orthonormal basis $\{e_i\}$ for $\br^{4m+4}$ in such a way that $o=\pi e_{4m+4}$ and that $J_\a e_{4m+4} = e_{4m+\a}$ for $\a=1,2,3$. At the point $e_{4m+4}$, the horizontal and the vertical subspaces of the Hopf fibration are defined, respectively, by $\mathscr{H}=\Span(e_1, \dots, e_{4m})$ and $\mathscr{V} = \mathscr{H}^\perp$. The subspace $\mathscr{H}$ can be identified with $T_o \bh P^m$, and the isotropy subalgebra at the point $o$ is given by $\spg(m) \oplus \spg(1) \subset \spg(m+1)$, where $\spg(m)$ consists of the matrices annihilating $\mathscr{V}$, and $\spg(1)$, of the matrices annihilating $\mathscr{H}$. Although the subalgebra $\spg(1)$ is abstractly isomorphic to $\Span(J_1,J_2,J_3)$, they have nothing in common (the restrictions of $\Span(J_1,J_2,J_3)$ to $\mathscr{V}$ and $\spg(1)$ are two complementary $\so(3)$ ideals in $\so(4)$). However, we can choose a basis $\{L_1, L_2, L_3\}$ for $\spg(1)$ in such a way that $L_\a e_{4m+4} = J_\a e_{4m+4}$ for $\a=1,2,3$. Relative to this basis, the action of an element $A' = A \oplus L_\a$ on $\mathscr{H}$, where $A \in \spg(m)$, is given by $A'(X) = AX - J_\a X$ for $X \in \mathscr{H}$. Now any element of $\sT^1_{m-1}$ is a linear combination of the tensors $T^1_{A,B}$ given by~\eqref{eq:hpmT1} for $A \odot B \in \Sym^2(\spg(m))$, and so defining the matrices $A'=A \oplus 0, B' = B \oplus 0 \in \spg(m) \oplus \spg(1) \subset \spg(m+1)$ we obtain that the tensor $T^1_{A',B'}$ defines a top slot quadratic Killing tensor field $F_{T^1_{A',B'}}$ on $\bh P^m$ whose corresponding $(0,4)$ tensor constructed as in Theorem~\ref{t:quadratic}\eqref{it:quadratictop1} coincides with $T^1_{A,B}$. Similar argument, with obvious modification, applies to the elements of $\sT^2(m-1)$. Now, for a tensor $K_2^{(1)}$ defined by $K_2^{(1)}(X,X,P,P) = \sum_{\b,\gamma=1}^{3} c_{\b \gamma} \<J_\b X, P\> \<J_\gamma X, P\>$, with $c_{\b \gamma} = c_{\gamma \b}$, we take the element $T =\sum_{\b,\gamma=1}^{3} c_{\b \gamma} T^1(0\oplus L_\b,0\oplus L_\gamma) \in \sT^1_m$. Then $F_T$ is a top slot quadratic Killing tensor field at $o$ whose corresponding tensor $\cK$ coincides with $K_2^{(1)}$. Finally, for the tensor $K_2^{(2)}$ defined by $K_2^{(2)}(X,X,P,P) = \sum_{\a=1}^{3} \<N_\a X, P\> \<J_\a X, P\>$, where $N_\a \in \spg(m)$, we take $T = - \sum_{\a=1}^{3} T^1(N_\a \oplus 0, 0 \oplus L_\a) \in \sT_m^1$.

Thus any top slot quadratic Killing tensor field on $\bh P^m$ equals $F_T$ given by~\eqref{eq:FT}, for some $T \in \sT_m$. By Theorem~\ref{t:quadratic}\eqref{it:quadraticspan}, the same is true for any quadratic Killing tensor field. This establishes assertion~\eqref{it:hpnop2hpn} of Theorem~\ref{t:hpnop2}.

\smallskip

For assertion~\eqref{it:hpnop2hp2} of Theorem~\ref{t:hpnop2} we additionally need to show that any quadratic tensor field $F_T$ on $\bh P^2$ with $T \in \sT_2^2$ lies in the span of the quadratic tensor fields $F_T$ with $T \in \sT_2^1$ (this is no longer true for $m \ge 2$, by~\cite[Theorem~1]{MN1}). This is equivalent to saying that $\sT_2^2 \subset \sT_2^1 + \Ker F$. We identify $\br^{12}$ with the quaternionic module $\bh^3$, and for an element $q \in \bh$, define $\rL(q), \rR(q): \bh \to \bh$ to be the operators of the left and the right multiplication by $q$, respectively. Choose the coordinates in such a way that $J_\a = \diag(\rL(\ir_\a), \rL(\ir_\a), \rL(\ir_\a))$, where $\ir_1, \ir_2$ and $\ir_3$ are the standard imaginary quaternions. Then the elements of $\spg(3)$ are the block matrices, with the $4 \times 4$ blocks $\rR(q_{ij}), \, i,j =1,2,3$, where $q_{ij} \in \bh$ and $q_{ji}= -\overline{q_{ij}}$. The elements of $V_3$ are also the block matrices, with the $4 \times 4$ blocks $\rR(q_{ij}), \, i,j =1,2,3$, with $q_{ij} \in \bh$ and $q_{ji}= \overline{q_{ij}}$ (note that $q_{ii} \in \br$). Up to conjugation by an element of $\Sp(3)$, we can assume that $S= \diag (a_1 I_4, a_2 I_4, a_3 I_4)$, where $a_i \in \br$. Let $\mu_1, \mu_2, \mu_3 \in \br$ be such that $\mu_1 + \mu_2 = a_1a_2, \, \mu_2 + \mu_3 = a_2a_3$ and $\mu_3 + \mu_1 = a_3a_1$. Then for $X=(x^{(1)}, x^{(2)}, x^{(3)})^t$ and $P=(p^{(1)}, p^{(2)}, p^{(3)})^t$, where $x^{(i)}, p^{(i)} \in \br^4 = \bh$ we obtain  from~\eqref{eq:hpmT2}:
  \begin{align*}
    T^2_{S,S}(X,P) & = \sum\nolimits_{\a=1}^{3} \Big(\sum\nolimits_{i=1}^{3} a_i \<\rL_\a x^{(i)}, p^{(i)}\>\Big)^2 \\
     & = \sum\nolimits_{i \ne j, \a} a_i a_j \<\rL_\a x^{(i)}, p^{(i)}\>\<\rL_\a x^{(j)}, p^{(j)}\> + \sum\nolimits_{\a, i} a_i^2 \<\rL_\a x^{(i)}, p^{(i)}\>^2 \\
     & = \sum\nolimits_{i \ne j, \a} (\mu_i + \mu_j) \<\rL_\a x^{(i)}, p^{(i)}\>\<\rL_\a x^{(j)}, p^{(j)}\> + \sum\nolimits_{\a, i=1}^3 a_i^2 \<\rL_\a x^{(i)}, p^{(i)}\>^2 \\
     & = \sum\nolimits_{\a=1}^3 \Big(\sum\nolimits_{i=1}^3 \mu_i \<\rL_\a x^{(i)}, p^{(i)}\>\Big)^2 + \sum\nolimits_{\a, i=1}^3 (a_i^2-2\mu_i) \<\rR_\a x^{(i)}, p^{(i)}\>^2 \\
     & = T^2_{I_{12},S'}(X,P) + \sum\nolimits_{\a,i=1}^3 (a_i^2-2\mu_i) T^1(A_{i\a}, A_{i\a})(X,P),
  \end{align*}
where $S' = \diag (\mu_1 I_4, \mu_2 I_4, \mu_3 I_4) \in V_3$ and where $A_{i\a} \in \spg(3)$ is the matrix whose $i$-th diagonal $4 \times 4$ block is $\rR_\a$ and all other entries are zeros (note that we used the fact that $\sum_{\a=1}^{3} \<\rL_\a z, z'\>^2 = \sum_{\a=1}^{3} \<\rR_\a z, z'\>^2$, for any $z,z' \in \bh$). But now $T_{I_{12},S'} \in \Ker F$ by Lemma~\ref{l:hpmprep}\eqref{it:hpmprepkerF}, and so $T^2(S,S) \in \Ker F + \sT_2^1$, as required.

\smallskip

For the proof of assertion~\eqref{it:hpnop2op2}, we briefly recall the construction in~\cite[Section~3]{MN1}. Denote $H_3(\bo)$ be the Albert algebra, the Jordan algebra of $3 \times 3$ Hermitian octonion matrices. Elements of $H_3(\bo)$ have the form
  \begin{equation} \label{eq:H3O} 
    A = \left(
        \begin{array}{ccc}
          r_1 & x_3^* & x_2^* \\
          x_3 & r_2 & x_1 \\
          x_2 & x_1^* & r_3 \\
        \end{array}
      \right), \qquad x_1, x_2,x_3 \in \bo, \; r_1,r_2,r_3 \in \br,
  \end{equation}
with the Jordan multiplication given by $A \circ B = \frac12(AB+BA)$. The automorphism group of $H_3(\bo)$ is the exceptional Lie group $\Ff$. Its action preserves the trace given by $\Tr(A)=r_1+r_2+r_3$, the squared norm given by $\|A\|^2=\Tr(A^2)$, and the determinant given by $\det(A)=r_1r_2r_3 + 2\Re(x_1x_2x_3) - r_1 \|x_1\|^2 - r_2 \|x_2\|^2 - r_3\|x_3\|^2$. The squared norm and the determinant define an inner product $\ip$ and a symmetric trilinear form $\Phi$ on $H_3(\bo)$, respectively.

The Cayley projective plane $\bo P^2 = \Ff/\Spin(9)$ is the submanifold of $H_3(\bo)$, with the induced metric, defined as follows:
  \begin{equation*}
    \bo P^2 = \{X \in H_3(\bo) \, | \, \Tr(X)=1,\, \Phi(A,X,X)=0, \; \text{for all } A \in H_3(\bo)\}.
  \end{equation*}
The group $\Ff$ acts transitively on $\bo P^2$, and so $\bo P^2$ is the orbit of the element 
  \begin{equation*}
        E=\left(
        \begin{array}{ccc}
          1 & 0 & 0 \\
          0 & 0 & 0 \\
          0 & 0 & 0 \\
        \end{array}
      \right) \in H_3(\bo).
  \end{equation*}
The tangent space to $\bo P^2$ at $E$ is given by
  \begin{equation} \label{eq:TEOP2}
        T_E\bo P^2=\left\{
        \left(\begin{array}{ccc}
          0 & y^* & z^* \\
          y & 0 & 0 \\
          z & 0 & 0 \\
        \end{array}
        \right)
        \, \Big| \, y, z \in \bo \right\}.
  \end{equation}
Denote $V$ the hyperplane $\Tr A= 0$ in $H_3(\bo)$. For any $A \in V$, we define the quadratic tensor field $K_A$ on $\bo P^2$ as follows: for $X \in \bo P^2$ and $Y,Z \in T_X \bo P^2$, set $K_A(Y,Z)=\Phi(Y,Z,A)$. The tensor fields $K_A,\, A \in V$, are Killing and indecomposable when $A \ne 0$, and in particular, the map $A \mapsto K_A$ is injective.

To apply Theorem~\ref{t:quadratic}, we need to determine the top slot quadratic Killing tensor fields at $E$. The decomposable ones come from $\Sym^2(\so(9))$ and are spanned by the symmetric squares of the top slot Killing vector fields given by~\eqref{eq:veven}. To find the indecomposable ones, we note that the reflection of $\bo P^2$ at the normal space at the point $E$ is a global isometry. Under this reflection, a tensor field $K_A, \, A \in V$, is mapped to the tensor field $K_{A'}$, where $A'$ is obtained from $A$ by changing the signs of $x_2$ and $x_3$ in~\eqref{eq:H3O} to the opposite. It follows that quadratic Killing tensor fields $K_A$ which are even at $E$ are given by $A \in V$ with $x_2 = x_3 = 0$ in~\eqref{eq:H3O}. The dimension of the space of such $K_A$ is $10$. Note that they are not yet top slot at $E$, but can be made such by adding quadratic forms on the Killing vector field which are odd at $E$, as given in~\eqref{eq:vodd}, to make the value of the resulting tensor field $K_A'$ at $E$ zero. Note that $K_A'$ are still linearly independent.

We can now identify the space $T_E \bo P^2$ given in~\eqref{eq:TEOP2} with the space $\br^{16} = \bo \oplus \bo$. We introduce $9$ symmetric operators $S_i, \, i=0, \dots 8$, defined by $I_0(X)=(x_1,-x_2), \; I_i(X)=(e_i x_2^*,x_1^*e_i)$, for $i=1, \dots, 8$, for $X=(x_1,x_2) \in \bo \oplus \bo$, where $\{e_i\}$ is an orthonormal basis for $\bo$. The matrices $S_i$ satisfy the equations $S_iS_j+S_jS_i = 2\delta_{ij} I_{16}$. The isotropy algebra $\so(9)$ is spanned by the $36$ products $S_iS_j, \, 0 \le i < j \le 8$. This gives us explicitly the space of the tensors $K_2$ on $\br^{16}$, as in Theorem~\ref{t:quadratic}, corresponding to the decomposable top slot quadratic Killing tensors. That space has dimension $666 = \frac12 \cdot 36 \cdot 37$ (so there are no quadratic relations between the top slot Killing vector fields). Furthermore, the curvature tensor at $E$ is given by $R(X,P)P = 3(\|P\|^2X-\<X,P\>P) + \sum_{i=0}^{8} (\<S_i X, P\> S_i P - \<S_i P,P\>S_iX)$. Now we find all the tensors $K_2$ satisfying the conditions of assertion~\eqref{it:quadraticeqs} of Theorem~\ref{t:quadratic}. This part of the proof is computer aided; by Remark~\ref{rem:Bianchi}, a tensor $K_2$ satisfying~\eqref{eq:quadraticsigma} has $16^2(16^2-1)/12 = 5440$ independent components, and then solving the system of linear equations~\eqref{eq:quadratic21} for $K_2$, we find that the space of solutions has dimension $676$. It follows that the space of quadratic Killing tensor fields which are top slot at $E$ is spanned by the decomposable ones and by the tensor fields $K_A'$; there are no others. By Theorem~\ref{t:quadratic}\eqref{it:quadraticspan}, it follows that any quadratic Killing tensor field on $\bo P^2$ is a linear combination of the decomposable ones and of the tensors $K_A$, as claimed.

\end{document}